\documentclass[11pt]{amsart} 
\usepackage[utf8]{inputenc}
\usepackage[T1]{fontenc}
\usepackage{lmodern}
\usepackage[english]{babel}

\usepackage[a4paper,vmargin={3.5cm,3.5cm},hmargin={2.5cm,2.5cm}]{geometry}
\usepackage[font={small,sf}, labelfont={sf,bf}, margin=1cm]{caption}

\usepackage[pdftex,colorlinks=true]{hyperref}
\usepackage[expansion]{microtype}
\usepackage[pdftex]{color,graphicx} 
\usepackage{amsmath}
\usepackage{amsfonts,amssymb,amsthm,ulem}
\usepackage{stackrel,framed}
\usepackage{subfig}
\usepackage{enumitem}    

\usepackage{mathtools}   
\usepackage{stmaryrd}    

\usepackage{physics}    
\usepackage{xfrac}       

\usepackage{csquotes}
\usepackage[numbers]{natbib}

\graphicspath{{images/}}

\theoremstyle{plain}

\newtheorem{coro}{Corollary}
\newtheorem{theo}[coro]{Theorem}

\newtheorem{lemm}[coro]{Lemma}

\newcommand{\R}{\mathbb{R}}

\newcommand{\N}{\mathbb{N}}

\newcommand{\Q}{\mathbb{Q}}

\newcommand{\kZ}{\mathcal{Z}}

\newcommand{\kT}{\mathcal{T}}
\newcommand{\rV}{\mathrm{V}}
\newcommand{\rE}{\mathrm{E}}
\newcommand{\rF}{\mathrm{F}}

\renewcommand{\leq}{\leqslant}
\renewcommand{\geq}{\geqslant}

\newcommand{\llangle}{\langle\hspace{-0.07cm}\langle}
\newcommand{\rrangle}{\rangle\hspace{-0.07cm}\rangle}

\newcommand{\pnuc}{\ensuremath{\mathbb{P}^{\nu,c}}}

\newcommand{\enuc}{\ensuremath{\mathbb{E}^{\nu,c}}}
\newcommand{\vnuc}{\ensuremath{\mathrm{Var}^{\nu,c}}}

\hypersetup{
    pdftitle    = {Ising Random tetravalent planar maps }}

\begin{document}
\normalem

\title[Ising model with external magnetic field on random planar maps]{ Ising model with external magnetic field on random planar maps: Critical exponents}
\author{Nicolas Tokka}
\address{MODAL'X, Université Paris-Nanterre, Nanterre, and IRIF, Université Paris-Cité, Paris, France.}
\email{n.tokka@parisnanterre.fr}

\begin{abstract}
We study the Ising model with an external magnetic field on random tetravalent planar maps and investigate its critical behavior. Explicit expressions for spontaneous magnetization and the susceptibility are computed and the critical exponents $\alpha=-1$ (third order phase transition), $\beta=\frac{1}{2}$ (spontaneous magnetization), $\gamma=2$ (susceptibility at zero external magnetic field)  and $\delta=5$ (magnetization at critical temperature) are derived. To do so, we study the asymptotic behavior of the partition function of the model in the case of a weak external magnetic field using analytic combinatorics.  
\end{abstract}

\maketitle


\section{Introduction}

	The Ising model was introduced in the 1920s by Lenz and Ising \cite{Ising} to study magnetism. It was first investigated in dimension one, for which no interesting behavior appeared. A few years latter, Onsager \cite{Onsager} proved that the model exhibits a phase transition in dimension 2. Since then, the model has been extensively studied on regular lattices by probabilists and physicists (see the survey by Duminil-Copin \cite{DuminilCopin}).

	More recently, the Ising model and its probabilistic aspects have been studied on planar maps, which are graphs embedded into surfaces. The latter have been used by physicists as generic models of random geometries in dimension 2 to study quantum gravity. In this line of work the Ising model and the random lattices are coupled, and two types of questions arise. The first is about the behavior of the Ising model itself: does it exhibit a phase transition at finite temperature? The second question is about the geometry of the underlying map and how it is impacted by the Ising model, especially at criticality.

The Ising model on planar maps without an external magnetic field was initially explored by Kazakov in \cite{Kazakov86} through the use of matrix models. Subsequent work on its combinatorial aspects was conducted by Bouttier, Di Francesco, and Guitter in \cite{BDG_07}, and by Bernardi and Bousquet-Mélou in \cite{Bernardi_BBM}. More recently, Albenque, Chen, Ménard, Turunen, and Schaeffer have examined the probabilistic aspects of random maps coupled with the Ising model, as discussed in \cite{ChenTurunen22, ChenTurunen20, IsingAMS, AlbenqueMenard, Turunen}.

The study of the Ising model on planar maps with an external magnetic field was addressed by Boulatov and Kazakov in \cite{KazakovBoulatov87}, where they derived critical exponents using matrix models. More recent investigations include work by Duits, Hayford, and Lee in \cite{DHL, Hayford}, as well as research on higher genus maps by Bousquet-Mélou, Carrance, and Louf in \cite{BCL}. For the planar case, related combinatorial enumerations are provided in \cite{BMS} and \cite{AMT-Bij}.

The purpose of this work is to rigorously recover the critical exponents ($\alpha=-1$, $\beta=\frac{1}{2}$, $\gamma=5$ and $\delta=5$, see below for details) derived by Boulatov and Kazakov using analytic combinatorics. To this end, we focus on the Ising model on tetravalent planar maps with an external field. \medskip

	\subsection*{Ising model on tetravent planar maps}

	A spin configuration on a planar map $\mathfrak{m}$ is a mapping on the set of its vertices $\sigma : \rV(\mathfrak{m}) \rightarrow \left\{\ominus,\oplus\right\}$. We say that an edge $\left\{u,v\right\}$ of $\mathfrak{m}$ is monochromatic if $\sigma(u)=\sigma(v)$, and is frustrated otherwise. The number of monochromatic edges of $(\mathfrak{m},\sigma)$ is denoted by $m(\mathfrak{m},\sigma)$, and the number of its negative and positive spins are respectively denoted by $\sigma_\ominus$ and $\sigma_\oplus$.

	We denote by $\kT_n$ the set of tetravalent planar maps of size $n$ -- that is planar maps with $n$ vertices of degree 4 -- endowed with a spin configuration (see Section \ref{sect: Definitions} for precise definitions). The partition function of the Ising model on tetravalent planar maps of size $n$ is defined as the following (finite) sum:
$$\kZ_n(\nu,c) := \sum_{(\mathfrak{m},\sigma)\in\kT_n}\nu^{m(\mathfrak{m},\sigma)} c^{\sigma_\oplus-\sigma_\ominus}.$$
Note that, writing $\nu=\exp(2\beta)$ and $c=\exp(h)$, we recover up to a proportional term, the usual partition function and Gibbs measure of the Ising model with inverse temperature $\beta$ and external magnetic field $h$: 
\begin{equation}\label{equa: Physical partition function}
{\sum_{(\mathfrak{m},\sigma)\in\kT_n}    \exp\left(\beta \sum_{\left\{v,v'\right\}\in \rE(\mathfrak{m})}{\sigma(v)\sigma(v')} +h \sum_{v\in \rV(\mathfrak{m})}{\sigma(v)}\right)}{\exp(2n\beta)},
\end{equation}
where $\rV(\mathfrak{m})$ and $\rE(\mathfrak{m})$ are the set of vertices and edges of the map $\mathfrak{m}$. \bigskip

This partition function $\kZ_n$ will play a major role in this paper. We use analytic combinatorics to derive properties of $\kZ_n$ via its associated generating function ${\kZ\left(\nu,c,z\right)\in\Q(\nu,c)\llbracket z\rrbracket}$ defined as follows:
$$\kZ(\nu,c,z) \coloneq \sum_{n\geq 1}\kZ_n(\nu,c)z^n.$$
We first study the singularities of $\kZ$, leading to the  asymptotic behavior of the partition function. In particular, we show that $\kZ_n$  exhibits a combinatorial phase transition at the critical point $\nu_\star\coloneq 4$ when there is no external magnetic field. This transition is of the same nature as for the Ising model on triangulations (see \cite{Bernardi_BBM, IsingAMS}). We show that as soon as a magnetic field is present, no such phase transition occurs:
\begin{theo}\label{theo: Asymptotic behavior of the coefficients of Z}
For any $\nu>0$, there exists $0<\varepsilon_\nu<1$ such that for all $c\in\left[1-\varepsilon_\nu,1+\varepsilon_\nu\right]$, there exists a non-zero explicit constant $\gimel(\nu,c)$ in such a way that, as $n\rightarrow\infty$,
\begin{equation}
\kZ_n(\nu,c)\sim 
\begin{dcases}
 	\gimel(\nu_\star,1)\cdot \mu_{\mu,1}^{\hspace{0.3cm}-n}\hspace{0.05cm} n^{-7/3} & \text{ for $(c,\nu)=(1,\nu_\star)$,}	\\[0.2cm]
 	\gimel(\nu,c)\cdot \mu_{\mu,c}^{\hspace{0.3cm}-n}\hspace{0.05cm} n^{-5/2} & \text{ otherwise,}
\end{dcases}
\end{equation}
where $\mu_{\nu,c}$ is the radius of convergence of the power series $z \mapsto \kZ(\nu,c,z)$. 
\end{theo}

Notice that our result is valid for values of $c$ close to $1$, meaning that we only consider weak external magnetic fields. This is only a technical restriction. The result should be true for all $c>0$, but we did not push our arguments for the sake of simplicity.

\subsection*{Critical exponents}

Our main results deal specifically with the critical exponents for the Ising model on tetravalent maps in presence of a magnetic field. These exponents are defined using the free energy of the model. In our setting, the finite volume free energy is defined by:
\begin{equation}\label{equa: Definition of the free energy F_n}
F_n(\nu,c) \coloneq \frac{1}{n}\log \kZ_n(\nu,c),
\end{equation} 
The free energy in the thermodynamic limit is then defined as usual as:
\begin{equation} \label{equa: Definition of the free energy F}
\rF(\nu,c)\coloneq\lim_{n\rightarrow\infty} F_n(\nu,c).
\end{equation}
We prove that the phase transition of the model is of the third order, or equivalently that the critical exponent $\alpha$ is equal to $-1$:

\begin{theo}\label{theo : Critical exponents - alpha}
Under the assumptions of Theorem \ref{theo: Asymptotic behavior of the coefficients of Z}, the thermodynamic limit of the free energy exists and it can be expressed as a function of the radius of convergence $\mu_{\nu,c}$ of $\kZ(\nu,c,z)$:
\begin{equation}\label{equa: Expresion of the free energy in function of the radius of convergence of the power series I}
\rF(\nu,c)=-\log{\mu_{\nu,c}}.
\end{equation} 
Furthermore, when there is no external magnetic field, the free energy $\rF(\nu,1)$ is three times differentiable, and its third derivative is continuous except at $\nu=\nu_\star$. 
\end{theo}

The magnetization $M_n$ and the susceptibility $\chi_n$ of the model are defined as follows:
\begin{equation}\label{equa: Definition Mn and Xn : physic}
M_n(\nu,c)\coloneq c \partial_c F_n(\nu,c)\text{, \hspace{0.5cm}}\chi_n(\nu,c)\coloneq\left(c \partial_c\right)^2 F_n(\nu,c),
\end{equation} 
where $\left(c \partial_c\right)^2$ is the operator $c \partial_c$ applied twice (applying the operator $c\partial_c$ is equivalent to applying the operator $\partial_h$ in the usual parametrization of the Ising model with a magnetic field $h$). 
In the thermodynamic limit, the magnetization and the susceptibility are usually defined by:
\begin{equation}\label{equa: Definition of magnetization and susceptibility of the thermodynamic limit}
M(\nu,c)\coloneq c \partial_c \rF(\nu,c)\text{, \hspace{0.5cm}}  \chi(\nu,c)\coloneq\left(c \partial_c\right)^2 \rF(\nu,c).
\end{equation} 
We show in Lemma \ref{lemm: Concergence of Mn} that the magnetization of the thermodynamic limit model $M$ is indeed the limits of the finite volume magnetization $M_n$. \medskip

We are able to compute the spontaneous magnetization $M_0$ in the thermodynamic limit, see Figure \ref{fig: Spontaneous magnetization and Susceptibility}. It is defined as follows:
$$M_0(\nu)\coloneq\lim_{c\rightarrow 1^+} M(\nu,c).$$

This quantity represents the residual magnetization induced by the environment after the external magnetic field has been reduced to zero. The classical notation $M_0$ is used because the external magnetic field is null when the variable $h$ is zero, which, under our change of variables, corresponds to $c=1$. The usual critical parameter of the model is defined as:
$$\inf{\left\{\nu>0 \hspace{0.1cm}\vert\hspace{0.1cm} M_0(\nu)>0 \right\}}.$$
It is related to Curie's temperature through our change of variables. We are able to establish an explicit formula for  $M_0(\nu)$, from which we deduce the critical exponent $\beta=\frac{1}{2}$.

\begin{theo}\label{theo: Critical exponents - beta}
Recall that $\nu_\star=4$. For every $\nu\in (0,\infty)$ one has:
$$M_0(\nu)= \frac{3\nu\sqrt{\nu^2-16}}{3\nu^2-8}\mathbf{1}_{\nu\geq \nu_\star}.$$
As a consequence when $\nu\rightarrow \nu_\star\hspace{0.03cm}^{+}$:
$$M_0(\nu) \sim \frac{6\sqrt{2}}{5}\sqrt{\frac{\nu}{4}-1}.$$
\end{theo}

Note that the physical critical parameter of Theorem \ref{theo: Critical exponents - beta} coincides exactly with the combinatorial critical parameter of Theorem \ref{theo: Asymptotic behavior of the coefficients of Z}. Furthermore, we are able to compute the critical behavior of the magnetization at the critical temperature with low external magnetic field, yielding $\delta=5$:

\begin{theo}\label{theo: Critical exponents - delta}
The thermodynamic limit of the magnetization at $\nu=\nu_\star$ has the following asymptotic behavior as $c\rightarrow 1^+$:
$$M(\nu_\star,c) \sim \frac{3}{5}2^{3/5}\cdot \left(c-1\right)^{1/5}.$$
\end{theo}

\begin{figure}[t]
  \centering
  \subfloat	{\includegraphics[scale=0.9, page=1]{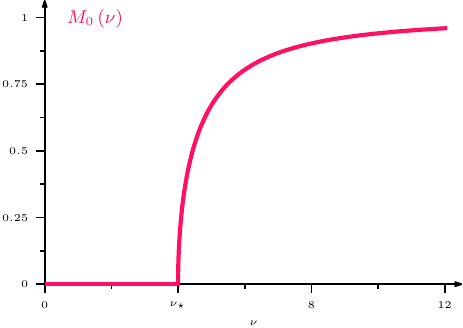}}\quad\quad
  \subfloat	{\includegraphics[scale=0.9, page=1]{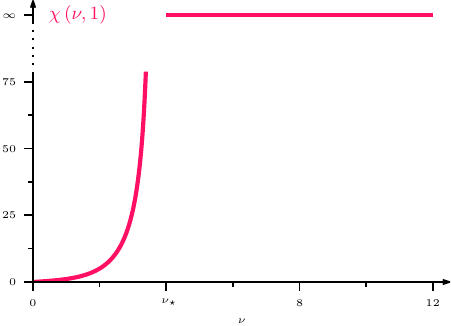}}
  \caption{The curves of the spontaneous magnetization $M_0(\nu)$ and of the susceptibility $\chi(\nu,1)$ as functions of $\nu>0$. A transition occurs at the critical point $\nu_\star=4$.}
  \label{fig: Spontaneous magnetization and Susceptibility}
\end{figure}

Finally, we turn to the susceptibility in the thermodynamic limit. Again we establish an explicit formula at $c=1$, allowing to derive our last critical exponent $\gamma=2$:

\begin{theo}\label{theo: Critical exponents - gamma}
The thermodynamic limit of thermodynamic limit of susceptibility satisfies
$$\chi(\nu,1) = 
\begin{cases}
	\frac{3\nu}{(2\sqrt{\nu}+1)(\sqrt{\nu}-2)^2} & \text{for $\nu\in (0,\nu_\star)$},\\
	\infty	& \text{for $\nu \geq \nu_\star$}.
\end{cases}$$
As a consequence when $\nu\rightarrow \nu_\star\hspace{0.03cm}^{-}$:
$$\chi(\nu,1) \sim \frac{12}{5\left(1-\frac{\nu}{4}\right)^2}.$$
\end{theo}

Note that all of the critical exponents we have computed verify the classical scaling relations (see for example \cite{DuminilCopin}) and agree with Boulatov and Kazakov \cite{KazakovBoulatov87}.

\subsection*{Organization of the paper}

We give precise definitions and prerequisites about planar maps and Ising model in Section \ref{sect: Preliminaries}. In Section \ref{sect: No external magnetic field}, we derive the asymptotic behavior of the coefficients of the partition function $\kZ$ when there is no external magnetic field. Finally, we study our model in the presence of an external magnetic field in Section \ref{sect: External magnetic field}, and prove our main results.

\subsection*{Acknowledgements}. We acknowledge support from ANR grant ProGraM (ANR-19-CE40-0025). We thank Marie Albenque and  Laurent Ménard for stimulating discussions  and their supervision.


\section{Preliminaries}\label{sect: Preliminaries}
	\subsection{Planar maps with spins : definitions and generating series} \label{sect: Definitions}
A \textit{planar map} is the proper embedding of a finite connected graph in the sphere $\mathbb{S}^2$ considered up to orientation-preserving homeomorphisms. Its \textit{faces} are the connected components of the complementary of the map on the sphere, and its \textit{edges} and \textit{vertices} correspond to their analogue on the graph.  The sets of faces, edges and vertices of a planar map $\mathfrak{m}$ are denoted by $\rF(\mathfrak{m})$, $\rE(\mathfrak{m})$ and $\rV(\mathfrak{m})$ respectively. A planar map can be \textit{rooted}, meaning that one of its edges is distinguished and oriented. This edge is called the \textit{root edge}, and the vertex at its tail is called the \textit{root vertex}.  Finally, a planar map is said to be \textit{tetravalent} if all its vertices have degree 4. 

\bigskip
A planar map $\mathfrak{m}$ can be endowed with a \textit{spin configuration}, which means that it is associated with a mapping $\sigma : \rV(\mathfrak{m}) \mapsto \left\{\ominus,\oplus\right\}$. For such a map $(\mathfrak{m},\sigma)$, an edge $\left\{u,v\right\}$ of is said to be \textit{monochromatic} if $\sigma(u)=\sigma(v)$, and \textit{frustrated} otherwise, see Figure \ref{fig: A Planar map with spins}. The number of monochromatic edges of $(\mathfrak{m},\sigma)$ is denoted by $m(\mathfrak{m},\sigma)$, and the number of negative and positive spins are respectively denoted by $\sigma_\ominus$ and $\sigma_\oplus$. All the maps we consider in this article are tetravalent rooted planar maps endowed with a spin configuration.

\bigskip

The set of tetravalent rooted planar maps endowed with a spin configuration and whose root vertex carries a positive spin,  is denoted by $\mathcal{T}$. Its subset composed of the maps with $n\geq 1$ vertices is denoted by $\mathcal{T}_n$. For any $n\geq 1$, the partition function of the tetravalent planar maps in $\mathcal{T}_n$ is defined as the following finite sum:
$$\kZ_n(\nu,c) \coloneq \sum_{(\mathfrak{m},\sigma)\in\kT_n}\nu^{m(\mathfrak{m},\sigma)} c^{\sigma_\oplus-\sigma_\ominus},$$
from which, the generating function of the tetravalent planar maps in $\mathcal{T}$ is defined in $\Q(\nu,c)\llbracket z\rrbracket$ as follows:
$$\kZ(\nu,c,z) \coloneq \sum_{n\geq 1}\kZ_n(\nu,c)z^n= \sum_{(\mathfrak{m},\sigma)\in\kT}\nu^{m(\mathfrak{m},\sigma)} c^{\sigma_\oplus-\sigma_\ominus}z^{\vert \rV(\mathfrak{m})\vert}.$$

We will also consider a model of random tetravalent planar maps endowed with a spin configuration. For any $n\in\N$, and fixed $\nu>0$ and $c>0$, let $\pnuc_n$ be the probability distribution supported on the elements of $\kT_n$, and defined by:
$$\pnuc_n(\left\{(\mathfrak{m},\sigma)\right\})\coloneq\frac{\nu^{m(\mathfrak{m},\sigma)} c^{\sigma_\oplus-\sigma_\oplus}}{\kZ_n(\nu,c)}.$$
When the change of variables used to obtain expression (\ref{equa: Physical partition function}) is applied, we obtain the traditional Gibbs measure of the Ising model.
The magnetization and the susceptibility of the finite model that have been defined in Equations (\ref{equa: Definition Mn and Xn : physic}) can be interpreted through this probability as follows:
\begin{equation}\label{equa: Definition Mn and Xn : proba}
M_n(\nu,c)=\frac{1}{n}\hspace{0.05cm}\enuc_n\hspace{-0.1cm}\left(\sum_{v\in \rV(\mathfrak{m})}{\sigma(v)}\right), \hspace{0.4cm}
\chi_n(\nu,c)=\frac{1}{n}\vnuc_n\hspace{-0.1cm}\left(\sum_{v\in \rV(\mathfrak{m})}{\sigma(v)}\right),
\end{equation} 
where $(\mathfrak{m},\sigma)$ is sampled following $\pnuc_n$, and  the symbols  $\oplus$, and $\ominus$ are interpreted as $+1$ and  $-1$, respectively. This point of view will be useful later.

\begin{figure}[t]
  \centering
  \includegraphics[page=1, scale =0.9]{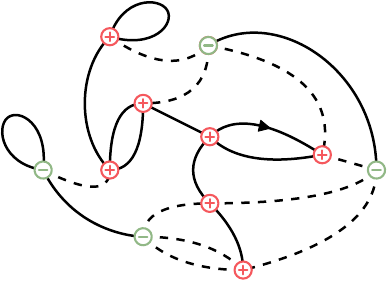}
  \caption{A tetravalent rooted planar map endowed with a spin configuration, whose weight is $\nu^{12}c^{3}z^{11}$ in the generating function $\kZ$. The frustrated edges are indicated with dotted lines.}
  \label{fig: A Planar map with spins}
\end{figure}

	\subsection{Planar maps with spins : Lagrangian parametrization}
	 
The results presented in this paper are based on a rigorous examination of the generating function $\kZ(\nu,c,z)$ of the tetravalent planar maps endowed with a spin configuration. To do so, we repeatedly use the following Lagrangian parametrization established in \cite{AMT-Bij}: 
	
\begin{theo}[Theorem 5.3, \cite{AMT-Bij}]\label{theo : Positivity generating function Ising tetra}
The generating function of the tetravalent planar maps in $\mathcal{T}$ can be expressed as follows:
\begin{equation}\label{equa: Parametrization of Z}
\kZ(\nu,c,cz)=\frac{ Pol_\kZ(S(\nu,c,z),\nu,c,z)}{9 z^2(1-\nu^2)(1+3c^2(1-\nu^2)S(\nu,c,z))}.
\end{equation}
where 
the series  $S\equiv S(\nu,c,z)$ is the unique formal power series in $\Q(\nu,c)\llbracket z\rrbracket$ with constant term 0 that satisfies the Lagrangian equation:
\small
\begin{equation}\label{equa: Definition of S}
z=\frac{S\left(1 - 3\nu^2(c^2 + 1)S- 3c^2(1-\nu^2)(3\nu^2 + 7)S^2 + 135c^4(1-\nu^2)^3S^4 - 243c^6(1-\nu^2)^5S^6 \right)}{\left(1-9c^2\left(1-\nu^2\right)^2S^2\right)^2}.
\end{equation}
\normalsize
and where the polynomial $Pol_\kZ(s,\nu,c,z) \in\mathbb{Q}\left[s,\nu,c,z\right]$ is defined as follows: 
\small
\begin{align*}
&Pol_\kZ(s,\nu,c,z) \\
&\quad\coloneq 405 c^{6} \left(1-\nu^{2}\right)^{4} s^{7}+351 c^{4} \left(1-\nu^{2}\right)^{3} s^{6}-27 c^{2} \left(1-\nu^{2}\right)^{2} \left(12 \left(1-\nu^{2}\right)^{2} c^{2} z +5 c^{2}-\nu^{2}\right) s^{5}\\
&\quad +3 c^{2} \left(1-\nu^{2}\right) \left(36 \left(1-\nu^{2}\right)^{2} c^{2} z -3 \nu^{2}-47\right) s^{4}+\left(252 \left(1-\nu^{2}\right)^{2} c^{2} z -\left(6 c^{2}+15\right) \nu^{2}-9 c^{2}\right) s^{3}\\
&\quad +\left(5-108 c^{2} \left(1-\nu^{2}\right)^{3} z^{2}+9 \left(1-\nu^{2}\right) \left(4 c^{2}+\nu^{2}\right) z \right) s^{2}-z \left(27 \left(1-\nu^{2}\right)^{2} c^{2} z -3 \nu^{2}+8\right) s\\
&\quad +3 z^{2} \left(1-\nu^{2}\right)
\end{align*}
\normalsize
Moreover, the series $S$ lies in $\mathbb{Z}\left[\nu,c\right]\llbracket z \rrbracket$ and has nonnegative coefficients. 
\end{theo}

This result is a direct reformulation in our settings of Theorem 5.4 of \cite{AMT-Bij}. 
Both parametrization are linked by the following change of variables:
\begin{equation*}
	\kZ(\nu,c,c\,t^2) = I(c^2,1,t,\nu,1) \quad\text{and}\quad S(\nu,c,t^2)=Q(c^2,1,t,\nu,1).
\end{equation*}

In \cite{AMT-Bij}, $I(x,y,t,\nu,u)$ denotes the weighted generating function of the maps in $\mathcal{T}$, and the weight $w(\mathfrak{m},\sigma)$ of a map endowed with a spin configuration $(\mathfrak{m},\sigma)$, is defined in ${\mathbb{Q}\left[x,y, t,\nu,u\right]}$ as follows:
\begin{equation*}
	w(\mathfrak{m},\sigma) \coloneqq x^{\sigma_\oplus} y^{\sigma_\ominus} t^{\vert\mathrm{E}(\mathfrak{m})\vert} \nu^{m(\mathfrak{m},\sigma)} u^{\vert\mathrm{F}(\mathfrak{m})\vert}.
\end{equation*}
This change of variables is based on the observation that, in a tetravalent map, the number of edges is twice the number of vertices. Additionally, in \cite{AMT-Bij}, spin configurations are represented by coloring the vertices in black and white, with black and white vertices corresponding to $\oplus$ and $\ominus$ spins, respectively, in the present work.

Finally, we emphasize that both the nonnegativity of the coefficients of $S$, and the Lagrangian canceling equation it satisfies play a crucial role in several proofs in this article.

\bigskip

To end this paragraph, we provide the first terms of the power series $\kZ$. The tetravalent rooted planar maps endowed with a spin configuration which are counted by the two first terms are represented in Figure \ref{fig: Small tetravalent planar maps endowed with the Ising model}.
$$\kZ(\nu,c,z) =  2\nu^2cz + (9\nu^4c^2 + 8\nu^2 + 1)z^2 + 18(3\nu^6c^3 + 4\nu^4c + 2\nu^2c + 2\nu^4c^{-1} + \nu^2c^{-1})z^3 + O(z^3).$$

\begin{figure}[t]
  \centering
  \includegraphics[page=2, scale =0.63]{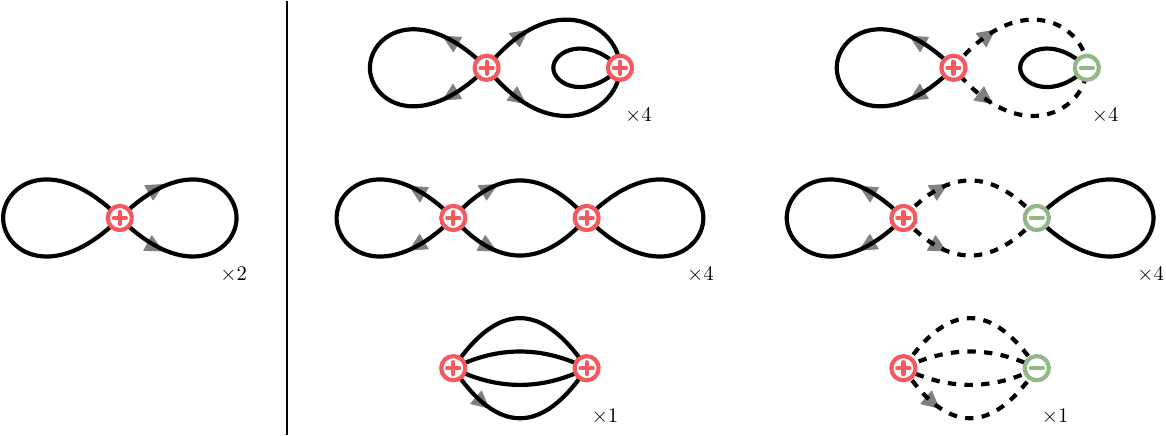}\;
  \caption{The unrooted tetravalent planar maps endowed with a spin configuration having one or two vertices. The possible rootings are given by the grey orientations. The frustrated edges are indicated by dotted lines. }
  \label{fig: Small tetravalent planar maps endowed with the Ising model}
\end{figure}

\section{Ising model without external magnetic field: \texorpdfstring{$c=1$}{c=1}}\label{sect: No external magnetic field}

The main goal of this section is to provide a complete study of the singularities of the power series $S$ and $\kZ$ when $c=1$, meaning that we consider the model with no external magnetic field. From now on, to lighten the notations we write $S(\nu,z)$ for $S(\nu,1,z)$ and $\kZ(\nu,z)$ for $\kZ(\nu,1,z)$. We begin with the study of the power series $S$.

\begin{lemm}\label{lemm : study of parameter S when c=1}
For any $\nu>0$, the radius of convergence $\rho_\nu$ of the power series $S(\nu,z)$ defined by (\ref{equa: Definition of S}) is finite and is its unique dominant singularity. It is explicitly given by:
\begin{equation}\label{equa: Expression of rho(nu,1)}
\rho_\nu =\begin{dcases}
 	\frac{2(1+2\sqrt{\nu})}{9(1+\sqrt{\nu})^2(1+\nu)^2} & \text{ for $\nu \in (0,\nu_\star)$, }\\[5pt]	
 	\frac{3\nu^2-8}{36(1-\nu^2)^2} & \text{ for $\nu \geq \nu_\star$,}
\end{dcases}
\end{equation} 
where $\nu_\star=4$. Moreover, the power series $S(\nu,z)$ converges at $\rho_\nu$ and there exists a positive explicit constant $\aleph(\nu)$ such that:
\begin{equation*}
S(\nu,z) =\begin{dcases}
 	S(\nu,\rho_\nu)- \aleph(\nu) \cdot(1-z/\rho_\nu)^{1/2}+o\big((1-z/\rho_\nu)^{1/2}\big) & \text{ for $\nu \neq \nu_\star$, }\\
 	S(\nu,\rho_{\nu_\star}) - \aleph(\nu_\star)\cdot(1-z/\rho_{\nu_\star})^{1/3}+o\big((1-z/\rho_{\nu_\star})^{1/3}\big) & \text{ for $\nu = \nu_\star$,}
\end{dcases}
\end{equation*} 
when $z\rightarrow\rho_\nu\hspace{0.03cm}^{-}$. The quantity $S(\nu,\rho_\nu)$ is given by:
\begin{equation}\label{equa: Expression of S(nu,1,rho)}
S(\nu,\rho_\nu) =\begin{dcases*}
 	\frac{1}{3(\sqrt{\nu}+1)(\nu+1)} & for $\nu\in(0,\nu_\star)$,\\
 	\frac{1}{3(\nu^2-1)} &  for $\nu\geq \nu_\star$.
\end{dcases*}
\end{equation}
\end{lemm}

\begin{proof}
All computations performed in this proof are detailed in the Maple companion \cite{Maple}. We first prove that $\rho_\nu$ is a singularity of $z\mapsto S(\nu,z)$, and then we calculate its explicit formula. For any $\nu>0$, the power series $S(\nu,z)$ has nonnegative coefficients (see Theorem \ref{theo : Positivity generating function Ising tetra}), so it is singular at its radius of convergence by Pringsheim's theorem (see \cite[Theorem IV.6, p.240]{FS}). To obtain an expression of $\rho_\nu$, we start by computing the quantity $S(\nu,\rho_\nu)$. It is clear that there exists a rational function $\phi(s)\coloneq \phi(\nu,c,s) \in \mathbb{Q}[\nu,c](s)$ such that (\ref{equa: Definition of S}) can be written as the following Lagrangean expression:
$$S=z\cdot\phi(S).$$ 
Furthermore, it is classical that $S(\nu,\rho_\nu)$ is the smallest real positive solution of the characteristic equation (see \cite[Definition VII.3, p.453]{FS}):
\begin{equation}
\phi(S)-S\phi'(S)=0
\end{equation}
Among all the solutions of the previous display, four are real. We obtain the expression of $S(\nu,\rho_\nu)$ given in System (\ref{equa: Expression of S(nu,1,rho)}) by identifying the smallest nonnegative one. First, note that for all $\nu>0$, we have that $S(\nu,\rho_\nu)\leq \frac{1}{3|\nu^2-1|}$. Secondly, observe that by plugging this explicit value of  $S(\nu,\rho_\nu)$ into Equation (\ref{equa: Definition of S}) we obtain the expression of $\rho_\nu$ announced in the statement. We immediately get that $\nu\in(0,\infty)\mapsto\rho_\nu\in(0,\infty)$ is a bounded continuous and decreasing function.

\bigskip

We now turn to the proof that $S(\nu,z)$ has a unique dominant singularity. The singularities of $S(\nu,z)$ are among the roots of the discriminant of any of its irreducible cancelling polynomial. Equation (\ref{equa: Definition of S}) provides such a polynomial (see the Maple companion \cite{Maple}). Its discriminant factorizes into a constant times a product of polynomials $P_1^3P_2P_3$, where
\begin{align}
P_1 (\nu,z) &= 36(\nu^2 - 1)^2z - 3\nu^2 + 8, \notag\\
P_2 (\nu,z) &= 81(\nu^2 - 1)^2(\nu + 1)^2z^2 + 36(3\nu - 1)(\nu + 1)^2z - 16\nu + 4, \notag\\
P_3 (\nu,z) &= P_2(-\nu,z) \notag.
\end{align}
Observe that $\rho_\nu$ is a root of $P_2$ when $\nu\in (0,\nu_\star]$, and a root of $P_1$ when $\nu>\nu_\star$. The other possible dominant singularities of S are among the roots of those polynomials that are different from $\rho_\nu$ but have the same modulus. In the following paragraphs, we prove that such roots are never singularities of $S$. For this task, and repeatedly thereafter, we use the following methodology (see \mbox{\cite[p.495-505]{FS}}):

\begin{framed}
\textbf{Methodology.} To study the behavior of an algebraic complex function $z\mapsto A(z)$ at a point $\widetilde{z}$ within its domain of definition, one can proceed as follows:
\begin{enumerate}[label=A.]
	\item First, compute an irreducible polynomial $\mathrm{Pol}(Z,Y)\in\left(\mathbb{C}[Z]\right)[Y]$ such that $$\mathrm{Pol}(\widetilde{z}-z, A(z))=0.$$
\end{enumerate}
The roots of $\mathrm{Pol}(Z,Y)$ with respect to the variable $Y$, are Puiseux series $y(Z)\in\mathbb{C}\llangle Z \rrangle$ that are convergent in a neighborhood of $Z=0$ (possibly with $0$ excluded). Thus, among those that verify the initial condition $y(0)=A(\widetilde{z})$, one corresponds to $A$. Note that sometimes we do not have direct access to an explicit expression for $\widetilde{z}$ and $A(\widetilde{z})$. To overcome this, one can use cancelling polynomials of these quantities and resultants, first to construct the polynomial $\mathrm{Pol}$ and then to verify the condition $y(0)=A(\widetilde{z})$.
\begin{enumerate}[label=B.]
	\item Then, for each solution $y(Z)$ of $\mathrm{Pol}(Z,y(Z))=0$, compute the first terms of their Puiseux expansion at $Z=0$, up to a sufficiently high order $K$ so they all have a distinct coefficient from the others.  To do so, one can apply Newton's polygon method on the polynomial $\mathrm{Pol}$.
\end{enumerate}
Let consider such a solution $y(Z)$, and designate the first terms of its Puiseux expansion up to $K$ by the following expression   $$y(Z)=y_{k_0}Z^{k_0/\kappa} + y_{k_0+1}Z^{(k_0+1)/\kappa} + \cdots + y_{K}Z^{K/\kappa} + {o}(Z^{K/\kappa}),$$ with $k_0\leq K$ in $\mathbb{Z}$ and $\kappa\geq 1$. If all the exponents of $Z$ up to $K/\kappa$, namely ${k_0}/{\kappa},\cdots,{K}/{\kappa}$, are nonnegative integers, then one can prove that $y(Z)$ is a formal power series in $Z$. Otherwise, $y(Z)$ admits a singularity at $Z=0$. Two cases are of interest to us:
\begin{enumerate}[label=$\text{C}_{1}$.]
	\item If one can distinguish which solution $y(Z)$ of $\mathrm{Pol}(Z,y(Z))=0$ corresponds to the function of interest $A$, then one obtain an explicit asymptotic behavior of $A(z)$ at the point $\widetilde{z}$. 
\end{enumerate}
\begin{enumerate}[label=$\text{C}_{2}$.]	
	\item If one can distinguish a subfamily of solutions of $\mathrm{Pol}(Z,y(Z))=0$ containing $A(z)$, and such that all of them admit the same asymptotic behavior (either no singularity, or a singularity with the same exponent), then one can conclude that $A(z)$ has this same asymptotic behavior at  the point  $\widetilde{z}$. 
\end{enumerate}
\end{framed}

We now apply this methodology to the roots of $P_1$, $P_2$ and $P_3$ starting with the two conjugate roots of $P_3$. Denote by $z_3$ any of them. For $0<\nu\leq \nu_\star$, the complex number $z_3$ have a modulus greater than $\rho_\nu$, so it can not be a dominant singularity of $S$. For $\nu> \nu_\star$, with a direct computation we get that $z_3$ has the same modulus than $\rho_\nu$ if and only if $\nu$ is a root of $$9\nu^3 - 256\nu^2 - 624\nu - 384.$$ This polynomial has a unique real root, which we denote by $\widetilde{\nu}$. Thus, $z_3(\widetilde{\nu})$ is a possible dominant singularity of $S$. To settle this case, we study the asymptotic behavior of $S(\widetilde{\nu},z)$ at $z=z_3(\widetilde{\nu})$ using the methodology described above. First, let us compute a cancelling polynomial for the power series $S$. It is clear that there exists a non-zero polynomial $\mathrm{Pol_1}(V,Z,Y)\in\Q[V,Z,Y]$ such that  (\ref{equa: Definition of S}) can be written as 
$$\mathrm{Pol}_1(\nu,z,S(\nu,z))=0.$$
Then, defining $\mathrm{Pol_2}\in\mathbb{C}[Z,Y]$ as $\mathrm{Pol_2}(Z,Y)=\mathrm{Pol}_1(\widetilde{\nu},z_3(\widetilde{\nu})-Z,Y)$, we obtain the following equality 
$$\mathrm{Pol_2}(z_3(\widetilde{\nu})-z,S(\widetilde{\nu},z))=0.$$
The polynomial $\mathrm{Pol_2}(Z,Y)$ has degree $6$ with respect to the variable $Y$. We compute for each of its root the first term of their Puiseux expansion at $Z=0$. All these expansions except one have a constant coefficient with modulus greater than $S(\widetilde{\nu},\rho_{\widetilde{\nu}})$. Since, the coefficients of $S$ are positive, we know that $|S(\nu,z)|\leq S(\nu,\rho_\nu)$ for any $|z|\leq\rho_\nu$ and $\nu>0$, so that these expansions cannot correspond to $S(\nu,z)$ around $\rho_{\widetilde{\nu}}$. We can restrict our attention to the last expansion, and using Newton's polygon method, we prove that it is not singular at $Z=0$. Thus, the roots of $P_3$ are not singularities of $S$. 

\bigskip

We now move on the roots of $P_2$. For $0<\nu\leq\nu_\star$, the radius of convergence $\rho_\nu$ is the positive root of $P_2$. By direct calculations, we know that its other root is on the circle of convergence of $S$ if and only if it is equal to $-\rho_\nu$, and that this happens only for $\nu={1/3}$. Once again, we prove that $-\rho_{{1/3}}$ is not a singularity of $S({1/3},z)$. First, we define $\mathrm{Pol_3}\in\mathbb{C}[Z,Y]$ as $\mathrm{Pol_3}(Z,Y)=\mathrm{Pol}_1({1/3},-\rho_{1/3}-Z,Y)$, so we get that $$\mathrm{Pol_3}(-\rho_{1/3}-z,S({1/3},z))=0.$$ 
Exactly as previously the polynomial $\mathrm{Pol_3}(Z,Y)$ has degree $6$ with respect to the variable $Y$. We compute for each of its root the first term of their Puiseux expansion at $Z=0$ and we observe that except for one of them, they all have a constant coefficient with modulus greater than $S(\widetilde{\nu},\rho_{\widetilde{\nu}})$. The remaining root corresponds to $S$ and it is not singular at $Z=0$. Thus, the complex number $-\rho_{{1/3}}$ is not a singularity of $S({1/3},z)$. 

For $\nu>\nu_\star$, the radius of convergence $\rho_\nu$ is the unique root of $P_1$. The positive root of $P_2$ is smaller that $\rho_\nu$ and we check that its negative root is equal to  $-\rho_\nu$ only for $\nu=\frac{4(4+\sqrt(34))^2}{9}$. We conclude that it is not a singularity of $S$ as before.

\bigskip

We end with the root of $P_1$. For $\nu\geq\nu_\star$, this root is equal to the radius of convergence. For $0<\nu<\nu_\star$ it lies on the circle of convergence if and only if it is equal to $-\rho_\nu$. Exactly as before, this happens at a specific value of $\nu$, here $\frac{4(-4+\sqrt(34))^2}{9}$. We conclude with similar arguments.

This concludes the proof that $S(\nu,\rho_\nu)$ has a unique dominant singularity for all $\nu>0$. 

\bigskip

To finish, we have to prove the asymptotic behavior of $S(\nu,z)$ announced in the lemma. To do this, we apply the methodology one more time. Fix $\nu>0$. The real numbers  $\rho_{\nu}$ and $S(\nu,\rho_\nu)$ have been explicitly computed before so we can define $\mathrm{Pol_4}(Z,Y)=\mathrm{Pol}_1(\nu,\rho_\nu-Z,S(\nu,\rho_{\nu})-Y)$ in $\R[Z,Y]$, which satisfies
$$\mathrm{Pol_4}(\rho_\nu-z,S(\nu,\rho_{\nu})-S(\nu,z))=0.$$
The polynomial $\mathrm{Pol_4}(Z,Y)$ factorizes in six terms but only one vanishes at $Z=Y=0$. Thus, we need to investigate the roots of this factor. We apply Newton's polygon method to compute  the singularity behavior of  its roots at $Z=0$. For $\nu\neq\nu_\star$, they all have a square root singularity. On the other hand, when $\nu=\nu_\star$ some coefficients of $\mathrm{Pol_4}$ vanishes and all the roots of interest have a singularity with exponent ${1/3}$ as expected. In each case, all the roots have the same constant coefficient in their Puiseux asymptotic expansion. Thus, we conclude that $S(\nu,z)$ have the same asymptotic behavior at $z=\rho_{\nu}$, and obviously the same constant term as them. 
\end{proof}

We derive the asymptotic behavior of the generating function $\kZ$  from the singular behavior of $S$:

\begin{coro}\label{coro : Asymptotic expansion of I when c=1}
The power series $\kZ(\nu,z)$ has a unique dominant singularity at  $\rho_\nu$, and it converges at this point. Moreover, there exist two non-zero explicit constants $\beth(\nu)$ and $\daleth(\nu)$ such that the power series $\kZ(\nu,z)$ has the following singular behavior at $\rho_\nu$ when $\nu \neq \nu_\star$:
$$\kZ(\nu,z) =
 	\kZ(\nu,\rho_\nu)+ \beth(\nu) \cdot(1-z/\rho_\nu) + \daleth(\nu)\cdot (1-z/\rho_\nu)^{3/2} +o\big((1-z/\rho_\nu)^{3/2}\big), $$
and the following one when $\nu = \nu_\star$:
$$\kZ(\nu,z) = \kZ(\nu_\star,\rho_{\nu_\star})+ \beth(\nu_\star)\cdot (1-z/\rho_{\nu_\star}) + \daleth(\nu_\star)\cdot (1-z/\rho_{\nu_\star})^{4/3} +o\big((1-z/\rho_{\nu_\star})^{4/3}\big).$$

\end{coro}

\begin{proof}
The properties of $\kZ(\nu,z)$ stated here derived from those of $S(\nu,z)$. First, the proof of Lemma \ref{lemm : study of parameter S when c=1} implies that $|S(\nu,z)|\leq \frac{1}{3(\nu^2-1)}$ for all $|z|\leq\rho_\nu$ and all $\nu>0$, with possible equality only if $z=\rho_\nu$, by the aperiodicity of $S$ and the Daffodil lemma (see \cite[Lemma IV.1, p.266]{FS}). Therefore, the series $S$ and $\frac{1}{(1+3(1-v^2)S)}$ share the same unique dominant singularity $\rho_\nu$. Consequently, the parametrization of $\kZ(\nu,z)$ given in (\ref{equa: Parametrization of Z}) implies that the singularities of the series $\kZ(\nu,z)$ are exactly those of $S(\nu,z)$.

\bigskip

We compute the singular behaviour of $\kZ(\nu,z)$ at $\rho_\nu$ by plugging in an explicit singular expression of $S(\nu,z)$ into its parametrization (\ref{equa: Parametrization of Z}). The asymptotic expansion provided by Lemma \ref{lemm : study of parameter S when c=1} is not sufficient. Nonetheless, we iterate Newton's polygon method on $\mathrm{Pol_4}$, from the proof of Lemma \ref{lemm : study of parameter S when c=1}, to compute the asymptotic expansion of $S(\nu,z)$ up to order 2, to get the announced asymptotic expansion. These computations are available in the companion Maple file \cite{Maple}. We end by checking that the explicit expressions for $\beth(\nu)$ and $\daleth(\nu)$ we obtain never vanish. 
\end{proof}

\section{Ising model with external magnetic field: \texorpdfstring{$c\neq 1$}{c<=1}}	\label{sect: External magnetic field}

In this section we turn our attention to the analysis of the Ising model with an external magnetic field. First, we provide a complete study of the asymptotic behavior of the partition function $\kZ_n$ when the external field is weak, meaning that the variable $c$ is close to $1$. The main idea is to derive results from the case without an external magnetic field by perturbing it. In a second step, we deduce some critical exponents of the model. 

We restrict ourselves to the case of small external magnetic field (i.e. when $c$ is close to $1$) for technical reasons, in particular to avoid multiplying tedious computations. Nonetheless, in the case of general external magnetic field (i.e. for any $c>0$), all the results should be preserved and provable without additional theoretical material.
 
	\subsection{Enumerative results and asymptotic behavior}

The purpose of this section is to prove Theorem \ref{theo: Asymptotic behavior of the coefficients of Z}. As previously, we first concentrate on the power series $S$, from which we will derive the result.

The first step of the proof, is to demonstrate that the two powers series $S$ and $\kZ$ have close singularities. Denote the radius of convergence of $z\mapsto S(\nu,c,z)$ and $z\mapsto Z(\nu,c,z)$ by $\rho_{\nu,c}$ and $\mu_{\nu,c}$, respectively. The first result is the following simple relation between their radius of convergence.

\begin{lemm}\label{lemm : S and Z share the same singularities}
Under the asumptions of Theorem \ref{theo: Asymptotic behavior of the coefficients of Z}, the power series $S(\nu,c,z)$ and $\kZ(\nu,c,cz)$ share the same singularities. Thus, we have that $\mu_{\nu,c}=c\cdot\rho_{\nu,c}$.
\end{lemm}
\begin{proof}
All computations performed in this proof are detailed in the Maple companion \cite{Maple}. Observe that from the parametrization of $\kZ$  given in Equation (\ref{equa: Parametrization of Z}), it is sufficient to prove that $\frac{1}{(1+3c^2(1-v^2)S)}$ and $S$ share the same singularities. To do this, similarly to the proof of Corollary \ref{coro : Asymptotic expansion of I when c=1}, we demonstrate that the modulus of the power series $S$ within its disk of convergence is smaller than $\frac{1}{3c^2|1-v^2|}$, for values of $c$ close enough to $1$.

The coefficients of S being positive, we have that $\vert S(\nu,c,z)\vert\leq S(\nu,c,\rho_{\nu,c})$ for any $\vert z\vert\leq \rho_{\nu,c}$. Thus, to conclude it is sufficient to prove that $S(\nu,c,\rho_{\nu,c})<\frac{1}{3c^2|1-v^2|}$. The quantity $S(\nu,c,\rho_{\nu,c})$ is the smallest real positive solution of the characteristic equation $\phi(S)-\nolinebreak S\phi'(S)=0$, where $\phi$ is defined by rewriting the Equation (\ref{equa: Definition of S}) as $S=z\cdot\phi(S)$. The purpose of the rest of the proof is to demonstrate that the characteristic equation admits at least one real solution in the interval $(0,\frac{1}{3c^2|1-v^2|})$. 

First, one can observe that the numerator and the denominator of the characteristic equation do not share any common root for any $\nu>0$, and any $c$ chosen close enough to $1$. This can be done by computing there resultant and seeing that it does not cancel when we fix the value of $\nu$ and impose  $c$ to take values close enough to $1$ (see Maple companion file \cite{Maple}). Hence, we can focus on the numerator of the characteristic equation. The latter factorizes into a product of two polynomials $Q_1(S,\nu,c,z)\cdot Q_2(S,\nu,c,z)$ where 
$$Q_1(S,\nu,c,z) = \left(1+3c(1-\nu^2)S\right)\left(1-3c(1-\nu^2)S\right),$$
and where $Q_2$ is of degree $8$ with respect to $S$. When $c< 1$, the roots of $Q_1$ have a modulus greater than the bound we need. Thus, we investigate the roots of $Q_2$ and prove that it has a positive root whose modulus is smaller than $\frac{1}{3c^2|1-v^2|}$. To do this, we use the effective Sturm sequence method (see \cite{Sturm}) to count the number of real roots of $Q_2$ within the interval $(0,\frac{1}{3c^2|1-v^2|}]$. 

Recall that the Sturm's theorem states that the number of distinct real roots of a polynomial $\mathrm{Pol}$ in a half-open interval $(a,b]$ is equal to the difference  $V_\mathrm{Pol}(a)-V_\mathrm{Pol}(b)$, where $V_\mathrm{Pol}(x)$ in the number of sign variations of the so called Sturm sequence of the polynomial $\mathrm{Pol}$ at $x$. The latter is a finite sequence that can be computed recursively from  the polynomial and the value of the points of interest. Let us apply it to $Q_2$.

Fix $\nu>0$.  There exists a constant $\varepsilon_\nu>0$, such that for any  $c\in\left[1-\varepsilon_\nu,1+\varepsilon_\nu\right]$, the number of sign variation of the Sturm sequence of $Q_2$ at $S=0$ is either $4$ or $5$ and its number of sign variation at $S=\frac{1}{3c^2|1-v^2|}$ is always one less smaller. Hence, there is exactly one root of $Q_2$ in the half-open interval of interest. Furthermore, by plugging  $\frac{1}{3c^2|1-v^2|}$ into the characteristic equation, we observe that it does not cancel, eventually by taking a smaller value of $\varepsilon_\nu$. This conclude the proof of the first point of the Lemma.

\bigskip

As a consequence, we  obtain directly the relation $\mu_{\nu,c}=c\rho_{\nu,c}$ between the radius of convergence of the power series $\kZ$ and $S$.
\end{proof}

In order to study the singular behavior of the power series $S$, we will crucially need the following result about the regularity of its radius of convergence.

\begin{lemm}\label{lemm : The radius of convergence are contiuous in the varaible c}
Under the asumptions of Theorem \ref{theo: Asymptotic behavior of the coefficients of Z},  the radius of convergence of $S(\nu,c,z)$ and $\kZ(\nu,c,z)$ are continuous with respect to the variable $c$.
\end{lemm}
\begin{proof}
We prove the continuity of $\rho_{\nu,c}$ by considering it as the radius of convergence of the power series $Z(\nu,c,cz)$. This fact has been proved in Lemma \ref{lemm : S and Z share the same singularities} when $c$ lies in the interval $\left[1-\varepsilon_\nu,1+\varepsilon_\nu\right]$. Hence, the similar regularity for the radius of convergence of the power series $\kZ(\nu,c,z)$ and $S(\nu,c,z)$ follows directly from the relation between their radius of convergence.

\bigskip

First, let us fix $c_1\leq c_2$ in $\left[1-\varepsilon_\nu,1+\varepsilon_\nu\right]$. Observe that for any $n\geq 1$, we have the following inequality:
$$0\leq {c_1}^n \kZ_n(\nu,c_1) = \sum_{(T,\sigma)\in\kT_n}\nu^{m(T,\sigma)} {c_1}^{2\sigma_\oplus} \leq \sum_{(T,\sigma)\in\kT_n}\nu^{m(T,\sigma)} {c_2}^{2\sigma_\oplus} = {c_2}^n \kZ_n(\nu,c_2) $$
Thus, we deduce that $\kZ(\nu, c_1, c_1z) \leq \kZ(\nu, c_2, c_2z)$, for all $z>0$. The quantities are possibly infinite. On the other hand, by using that their are more vertices than positive spins in a planar map endowed with a spin configuration, we obtain that
\begin{align*}
 \kZ(\nu, c_2, c_2z)	&= \sum_{(T,\sigma)\in\kT}\nu^{m(T,\sigma)} {c_2}^{2\sigma_\oplus}z^{\vert \rV(T)\vert} \\
						&= \sum_{(T,\sigma)\in\kT}\nu^{m(T,\sigma)} {c_1}^{2\sigma_\oplus}z^{\vert \rV(T)\vert}  \left(\frac{c_2}{c_1}\right)^{2\sigma_\oplus} \\
						&\leq \sum_{(T,\sigma)\in\kT}\nu^{m(T,\sigma)} {c_1}^{2\sigma_\oplus}z^{\vert \rV(T)\vert}  \left(\frac{c_2}{c_1}\right)^{2\vert \rV(T)\vert} =  \kZ\left(\nu, c_1, \left(\frac{c_2}{c_1}\right)^2  c_1z\right),
\end{align*}
for all $z>0$. Therefore, we obtain the inequalities
$$\rho_{\nu,c_1} \geq \rho_{\nu,c_2} \geq \left(\frac{c_1}{c_2}\right)^2 \rho_{\nu,c_1}, $$
which proves the continuity of $c\mapsto \rho_{\nu,c}$ on $\left[1-\varepsilon_\nu,1+\varepsilon_\nu\right]$.
\end{proof}

From the previous lemma, we can provide a study of the singularities of the power series $S$.	
	
\begin{lemm}\label{lemm : study of parameter S for c close to 1}
Under the asumptions of Theorem \ref{theo: Asymptotic behavior of the coefficients of Z}, the  unique dominant singularity of the power series $S(\nu,c,z)$ is it radius of convergence $\rho_{\nu,c}$.
\end{lemm}
\begin{proof}
The case $c=1$ is the purpose of Lemma \ref{lemm : study of parameter S when c=1}. In this proof, we focus on the case $c\neq 1$. 

The power series $S(\nu,c,z)$ has nonnegative coefficients so that its radius of convergence is a singular point by Pringsheim's theorem. We prove that there is no other singular points on its circle of convergence, when $c\neq 1$ is close enough to $1$. 

Let us localize the possible singularities of the power series $S$. To do this, let $\mathcal{D}_{\nu,c}(z)$ be the discriminant of the polynomial equation verified by $S(\nu,c,z)$ and obtained from Equation (\ref{equa: Definition of S}). The singular points of $S$ are among the roots of $\mathcal{D}_{\nu,c}$. Thus, we show that there is exactly one root of $\mathcal{D}_{\nu,c}$ on the circle of convergence of the power series $S$, when $c\neq 1$ is taken close to $1$. 

\bigskip

The discriminant $\mathcal{D}_{\nu,c}$ has degree $8$, and recall from Lemma \ref{lemm : study of parameter S when c=1} that $\mathcal{D}_{\nu,1}(z)$ factorizes into a product of three polynomials $P_1^3P_2P_3$. We denote the roots of  $\mathcal{D}_{\nu,c}$ by $R_{1,1}$,$R_{1,2}$,$R_{1,3}$,$R_{1,4}$, $R_{2,1}$, $R_{2,2}$, $R_{3,1}$, and $R_{3,2}$, such that $P_1$ cancels at $R_{1,i}(\nu,1)$, for $i\in\left\{1,2,3,4\right\}$,  $P_2$ cancels at $R_{2,i}(\nu,1)$, for $i\in\left\{1,2\right\}$, and $P_3$ cancels at $R_{3,i}(\nu,1)$, for $i\in\left\{1,2\right\}$.

From now on, we fix $\nu>0$. Each root of $\mathcal{D}_{\nu,c}(z)$ is a  Puiseux series in the variable $c$, that is convergent in a neighborhood of $c=1$.   Moreover, from Newton's polygon algorithm, we get that the roots $R_{1,c}(\nu,c)$ are power series in $\sqrt{1-c}$, where we consider the principal square root, and that the roots $R_{2,i}(\nu,c)$ and $R_{3,i}(\nu,c)$ are power series in $(1-c)$, see the Maple companion file \cite{Maple}. Thus,  at $c=1$, all of them are analytic functions with radius of convergence greater than a positive constant $r_\nu$.

\bigskip

Assume that $\nu$ is a generic value, meaning that only one root of $P_1,P_2$ and $P_3$ lies on the circle of convergence of the series $S(\nu,1)$.

First, if $\nu<4$, then the radius of convergence of $S(\nu,1)$ is a simple root of $P_2$. Without loss of generality let's consider that it corresponds to $R_{2,1}(\nu,1)$. Each other root of $\mathcal{D}_{\nu,c}$ takes a distinct value at $c=1$, so that at this value of $c$,  the minimum distance from any of them to the circle of convergence of $S(\nu,1)$ is a positive constant $\delta_\nu>0$. Thus, by continuity of the Puiseux series, there exists a real number $\varepsilon_\nu\in (0,r_\nu]$, such that for all $c\in\left[1-\varepsilon_\nu,1+\varepsilon_\nu\right]$,
$$\Big\vert |B(\nu,c)|-|B(\nu,1)|\Big\vert<\frac{\delta_\nu}{2} ,$$
where $B$ is any of the roots of $\mathcal{D}_{\nu,c}$. Thus, if $B\neq R_{2,1}$, by the triangular inequality, for all $c\in\left[1-\varepsilon_\nu,1+\varepsilon_\nu\right]$
$$\Big\vert |B(\nu,c)|-R_{2,1}(\nu,c)\Big\vert>0 .$$
Thus, except for $R_{2,1}$, the roots of the discriminant, seen as functions of the variable $c$, stay away in modulus from $R_{2,1}$ when $c$ is close to $1$. Hence, by continuity of $c\mapsto \rho_{\nu,c}$ (see lemma \ref{lemm : The radius of convergence are contiuous in the varaible c}) we conclude that  $R_{2,1}(\nu,c)=\rho_{\nu,c}$ locally around $c=1$. Clearly we proved that there is only one root of $\mathcal{D}_{\nu,c}$ whose modulus is equal to the radius of convergence of $S$, for any  $c\in\left[1-\varepsilon_\nu,1+\varepsilon_\nu\right]$. 

Secondly, if $\nu>4$, then the radius of convergence of $S$ is the root of $P_1$. Observe that at $c=1$ the four roots $R_{1,1}$,$R_{1,2}$,$R_{1,3}$ and $R_{1,4}$ merge, and are equal to the radius of convergence of $S$.  The other roots of $\mathcal{D}_{\nu,c}$ all take a distinct values at $c=1$. With similar arguments as before, we have that they stay stay away from the circle of convergence of the power series $S(\nu,c)$ for any $c$ close enough to $1$. To conclude this case, we prove that the four quantities $R_{1,1}$,$R_{1,2}$,$R_{1,3}$,$R_{1,4}$ have distinct modulus for $c\neq 1$ and close enough to $1$, so that by continuity only one root corresponds to the radius of convergence of $S(\nu,c)$. As explained previously, each of these roots admits a Puiseux series expansion at $c=1$ of the form 
$$R_{1,i}(\nu,c)=\sum_{k\geq0}{b_{k,i}(\nu)\left(\sqrt{1-c}\right)^{k}}.$$ 
Any of these coefficients can be explicitly computed by Newton's polygons method. Furthermore, we derive from this expression that the modulus $|R_{1,i}(\nu,c)|$ admits an asymptotic expansion at $c=1$ of the form 
$$|R_{1,i}(\nu,c)|=\mathfrak{b}_{0,i}(v)+\mathfrak{b}_{2,i}(v)(1-c)+\mathfrak{b}_{3,i}(v)\left(\sqrt{1-c}\right)^{3}+o((\sqrt{1-c})^3),$$ where for any $k$, the coefficient $\mathfrak{b}_{k,i}$ is an explicit real function of the coefficients $b_{0,i},\dots,b_{k,i}$. One can compute explicitly these three first terms for each root (see the companion Maple file \cite{Maple}), and deduce that the asymptotic behavior of the $|R_{1,i}(\nu,c)|$ at $c=1$ are distinct from each other. Thus, the modulus of the branches are distinct when $c$ is close enough to $1$. This proves that only one root of the discriminant of $S$ lies on its circle of convergence.

\bigskip

Now, assume that  $\nu$ is an atypical value, meaning that at least two roots of $P_1,P_2$ and $P_3$ lie on the circle of convergence of the series $S$. This only happens at the following five explicit reals (see the proof of Lemma \ref{lemm : study of parameter S when c=1}): 
\begin{center}
\scalebox{0.9}{${1/3}, \frac{4\left(-4+\sqrt{34}\right)^2}{9}, 4,   \frac{4\left(369793 + 810\sqrt{359}\right)^{1/3}}{27}  + \frac{20596}{27\left(369793 + 810\sqrt{359}\right)^{1/3}}  +\frac{256}{27}   ,  \frac{4\left(4+\sqrt{34}\right)^2}{9}$ }
\end{center}
For each case, we prove that the modulus of the roots of  $\mathcal{D}_{\nu,c}$ that lie on the circle of convergence at $c=1$, have distinct modulus when $c\neq 1$ is close enough to $1$. We proceed as previously by computing the first terms of the asymptotic expansion of their modulus at $c=1$  (see the companion Maple file \cite{Maple}). In practice, it suffices to compute the first three terms to conclude. The roots of  $\mathcal{D}_{\nu,c}$ that do not lie on the circle of convergence at $c=1$, stay away from it when $c\neq 1$ is close enough to $1$ similarly as before. Once again, this proves that only one root of the discriminant of $S$ lies on its circle of convergence.
\end{proof}

Finally, we can deduce the asymptotic behavior of the partition function $\kZ$ and consequently prove Theorem \ref{theo: Asymptotic behavior of the coefficients of Z}.

\begin{proof}[Proof of Theorem \ref{theo: Asymptotic behavior of the coefficients of Z}]

\bigskip
When $c=1$, the results follows from  Corollary \ref{coro : Asymptotic expansion of I when c=1} and from the transfer theorem (see \cite[Theorem VI.3, p.390]{FS}). Fix $c\neq 1$. From Lemma \ref{lemm : S and Z share the same singularities} and Lemma \ref{lemm : study of parameter S for c close to 1}, we get that the power series $\kZ(\nu,c,z)$ has a unique dominant singularity, which is its radius of convergence. 

\medskip

To conclude, we first compute the singular behavior of $S$ at its radius of convergence. The following computations are detailed in the Maple companion \cite{Maple}. Fix $\nu>0$. As previously stated, rewriting Equation (\ref{equa: Definition of S}) as $S=z\cdot\phi(S)$, one have the equality $\phi(S)- S\phi'(S)=0$ for $S=S(\nu,c,\rho_{\nu,c})$. Hence, we can see $z=S\cdot\phi(S)^{-1}$ as a function of $S$ and  use the previous equations to compute its asymptotic behavior when $S\rightarrow {S(\nu,c,\rho_{\nu,c}})^-$. To do so, we use the methodology described in the proof of Lemma \ref{lemm : study of parameter S when c=1}. Then, by local inversion of $z=S\cdot\phi(S)^{-1}$, one can obtain an asymptotic behavior of $S$ seen as a function of $z$ when $z\rightarrow{\rho_{\nu,c}}^-$. By plugging this expression into the parametrization of $\kZ$ with $S$ in (\ref{equa: Parametrization of Z}), we obtain the following asymptotic behavior of the power series $\kZ$ when $z\rightarrow {\mu_{\nu,c}}^-$:
$$\kZ(\nu,c,z) =
 	\kZ(\nu,c,\mu_{\nu,c})+ \beth(\nu,c) \cdot(1-z/\mu_{\nu,c}) + \daleth(\nu,c)\cdot (1-z/\mu_{\nu,c})^{3/2} +o\big((1-z/\mu_{\nu,c})^{3/2}\big), $$
where $\beth(\nu,c)$, $\daleth(\nu,c)$ are explicit constants. Moreover, the quantities $\daleth(\nu,1)$, $\partial_c\daleth(\nu,1)$ and $\partial_c^2\daleth(\nu,1)$ exist and at least one is non-zero, depending of the value of $\nu$. Hence, $\daleth(\nu,c)\neq 0$ in a neighborhood of $c=1$ so that the previous asymptotic expansion is valid for any $\nu>0$ and for any $c\neq 1$ in this neighborhood. The explicit asymptotic behavior of the coefficients of $\kZ$ is a direct consequence of the transfer theorem (see \cite[Theorem VI.3, p.390]{FS}). 
\end{proof}

	\subsection{Critical exponents}

In this last section, we compute some critical exponents of the Ising model. First, we focus on the free energy, and prove Theorem \ref{theo : Critical exponents - alpha}. The phase transition of the model is of order 3.

\begin{proof}[Proof of Theorem \ref{theo : Critical exponents - alpha}]
The expression of the free energy given in (\ref{equa: Expresion of the free energy in function of the radius of convergence of the power series I}) follows directly from its definition given in (\ref{equa: Definition of the free energy F_n}) and (\ref{equa: Definition of the free energy F}), and from the asymptotic of the coefficients of the power series $\kZ$ given in Theorem \ref{theo: Asymptotic behavior of the coefficients of Z}. Its differentiability properties follow directly from this expression and from the explicit expression of the radius of convergence of $\kZ(\nu,1,z)$ given in Equation (\ref{equa: Expression of rho(nu,1)}).
\end{proof}

To compute other critical exponents, we first prove that the magnetization of the discrete model converges to the magnetization of the thermodynamic limit of the model.

\begin{lemm}\label{lemm: Concergence of Mn}
For any $\nu>0$, there exists $0<\varepsilon_\nu<1$ such that for all $c\in\left[1-\varepsilon_\nu,1+\varepsilon_\nu\right]$, the magnetization of the thermodynamic limit verifies:
$$M(\nu,c)=\lim_{n\rightarrow\infty} M_n(\nu,c)$$
\end{lemm}
\begin{proof}
Recall that the discrete magnetization $M_n$ (resp. the magnetization $M$) is defined as the derivative with respect to the variable $c$, of the discrete free energy $F_n$  (resp. the free energy $F$), see Expressions (\ref{equa: Definition Mn and Xn : physic}) and (\ref{equa: Definition of magnetization and susceptibility of the thermodynamic limit}). Moreover, the free energy $F$ is defined as the pointwise limit of discrete free energy $F_n$, as $n$ tends to infinity, see Expression (\ref{equa: Definition of the free energy F}). In the next paragraph, we prove that for any $\nu>0$, the sequence of functions $(M_n\left(\nu,\cdot\right))_n$ converges uniformly on $\left[1-\varepsilon_\nu,1+\varepsilon_\nu\right]$, so that, by limit-derivative inversion and uniqueness of the limit, we obtain that $(M_n\left(\nu,\cdot\right))_n$ converges pointwise to $M(\nu,\cdot)$ on this interval. Note that one could prove the result with a direct computation. Instead, we prefer to proceed by using some properties of the statistics involved to shorten the proof. 

Fix $\nu>0$. We prove the uniform convergence of $(M_n\left(\nu,\cdot\right))_n$, as $n$ tends to infinity, using the second Dini's theorem. First, for any $n\geq 1$, $M_n(\nu,\cdot)$ is a non decreasing function because its derivative $c\mapsto \frac{1}{c}\chi_n(\nu,c)$ is non negative on $(0,\infty)$, as $\chi_n(\nu,c)$ is the variance of a real random variable, see Equation (\ref{equa: Definition Mn and Xn : proba}). To conclude, we have to prove that $(M_n(\nu,\cdot))_n$ converges pointwise to a continuous function. Fix $c>0$. Observe that $M_n(\nu,c)$ can be written as follows:
\begin{equation}\label{equa: Expression of the magnetization with Z and diffcZ}
M_n(\nu,c)=c\partial_c F_n(\nu,c) = \frac{c}{n}\partial_c \log(\kZ_n(\nu,c)) = \frac{c\partial_c\kZ_n(\nu,c)}{n\kZ_n(\nu,c)},
\end{equation} 
Thus, one can obtain its convergence as $n$ tends to infinity by studying the asymptotic behavior of $\partial_c\kZ_n(\nu,c)$. To proceed, we express the power series $\partial_c\kZ(\nu,c,cz)$ as a rational function of the variables $\nu,c,z$ and $S$ by computing the derivative with respect to $c$ of Equations (\ref{equa: Definition of S}) and (\ref{equa: Parametrization of Z}) that respectively defines the formal power series $S$ and $\kZ$ (see the Maple companion \cite{Maple}). Then, following the proof of Theorem \ref{theo: Asymptotic behavior of the coefficients of Z}, we plug an asymptotic expansion of the power series $S(\nu,c,z)$, as $z\rightarrow\rho_{\nu,c}\hspace{0.03cm}^{-}$, in the parametrization of $\partial_c\kZ(\nu,c,z)$ to obtain the first terms of the asymptotic behavior of the power series $\partial_c\kZ(\nu,c,z)$ when $z\rightarrow {\mu_{\nu,c}}^-$. Then, by applying the transfer theorem (see \cite[Theorem VI.3, p.390]{FS}) we obtain, as $n\rightarrow\infty$:
\begin{equation*}
\left[z^n\right]\partial_c\kZ(\nu,c,z)\sim 
\begin{dcases}
 	\widetilde\gimel(\nu,1)\cdot \mu_{\mu,1}^{\hspace{0.3cm}-n}\hspace{0.05cm} n^{-4/3} & \text{ for $(c,\nu)=(1,\nu_\star)$,}	\\[0.2cm]
 	\widetilde\gimel(\nu,c)\cdot \mu_{\mu,c}^{\hspace{0.3cm}-n}\hspace{0.05cm} n^{-3/2} & \text{ otherwise.}
\end{dcases}
\end{equation*}
where $\widetilde\gimel(\nu,c)$ is a real number. One can observe that there exists $0<\varepsilon_\nu<1$ such that $\widetilde\gimel(\nu,\cdot)$ is a non-zero continuous function on $\left[1-\varepsilon_\nu,1+\varepsilon_\nu\right]$. It comes from the computation of the first and second derivative of $\widetilde\gimel(\nu,\cdot)$ at $c=1$, following the proof of Theorem \ref{theo: Asymptotic behavior of the coefficients of Z}. Thus, we derive from  Equation (\ref{equa: Expression of the magnetization with Z and diffcZ}) and from the asymptotic behavior of $\left[z^n\right]\partial_c\kZ(\nu,c,z)$ and $\left[z^n\right]\kZ(\nu,c,z)$ (see Theorem \ref{theo: Asymptotic behavior of the coefficients of Z}) that $M_n(\nu,c)$ converges as $n$ goes to infinity. Furthermore, its limit is a continuous function with respect to $c$, as it is proportional to the quotient of the non-zero and continuous functions $\widetilde\gimel(\nu,c)$ and $\gimel(\nu,c)$. We conclude using the second Dini's theorem. 
\end{proof}

We finally have all the tools we need to establish our critical exponents: $\delta=5$, $\beta=\frac{1}{2}$ and $\gamma=2$.

\begin{proof}[Proof of Theorem \ref{theo: Critical exponents - beta},\ref{theo: Critical exponents - delta}, and \ref{theo: Critical exponents - gamma}]
All computations performed in this proof are detailed in the Maple companion \cite{Maple}. First, we compute the asymptotic behavior of $M(\nu_\star,c)$ when $c\rightarrow 1^+$. To do so, we apply the same methodology as in the proof of Lemma \ref{lemm : study of parameter S when c=1}. We calculate a cancelling polynomial for $M(\nu_\star,c)$, then we compute the Puiseux expansion of its roots when $c\rightarrow 1$, and finally we derive from them the announced asymptotic behavior.
Recall, that the discriminant $\mathcal{D}_{\nu,c}(z)$ of the polynomial equation verified by $S(\nu,c,z)$ obtained from Equation (\ref{equa: Parametrization of Z}) cancels its radius of convergence $\rho_{\nu,c}$. Hence, the polynomial $\widetilde{\mathcal{D}}_{\nu,c}(z)\coloneq c^{16}\cdot\mathcal{D}_{\nu,c}(\frac{z}{c})\in\mathbb{Q}[\nu,c][z]$  verifies that $\widetilde{\mathcal{D}}_{\nu,c}(\mu_{\nu,c})=0$. Thus, differentiating the latter and using Equation (\ref{equa: Expresion of the free energy in function of the radius of convergence of the power series I}), one can express:
$$M(\nu,c)=\eval{\frac{c\partial_c\widetilde{\mathcal{D}}_{\nu,c}(z)}{z\partial_z\widetilde{\mathcal{D}}_{\nu,c}(z)}}_{z=\mu_{\nu,c}}.$$
From this expression, it follows that the polynomial $\mathrm{Pol}^1_{\nu,c}(Y)\in\mathbb{Q}\left[ \nu,c\right][Y]$ defined below, is of degree $8$ and cancels at $Y=M(\nu,c)$:
$$\mathrm{Pol}^1_{\nu,c}(Y)\coloneq\mathrm{Resultant}_X(Y\cdot X\partial_z\widetilde{ \mathcal{D}}_{\nu,c}(X)-c\partial_c\widetilde{\mathcal{D}}_{\nu,c}(X)   ,  \widetilde{\mathcal{D}}_{\nu,c}(X) ).$$

Fix $\nu=\nu_\star$. We compute the first term of the Puiseux expansion of its roots at $c=1$. We have to distinguish which one corresponds to  $M(\nu_\star,c)$. Observe that $c\mapsto M_n(\nu_\star,c)$ is a non decreasing function because its derivative $c\mapsto \frac{1}{c}\chi_n(\nu_\star,c)$ is non negative on $(0,\infty)$, as $\chi_n(\nu,c)$ is is the variance of a real variable. Thus, by Lemma \ref{lemm: Concergence of Mn}, $c\mapsto M(\nu_\star,c)$ inherits the non  decreasing property. This is verified for only one root of $\mathrm{Pol}^1_{\nu_\star,c}$, so it corresponds to $M(\nu_\star,c)$. Its Puiseux expansion at  $c=1$ gives the announced asymptotic behavior.

\bigskip

Now, let us compute $M(\nu,1)$. The previous method works here, but discriminating the roots of the resulting cancelling polynomial is a bit tedious. Nevertheless, we can compute the spontaneous magnetization directly from the information we know about the radius of convergence of $S$. From Equation (\ref{equa: Expresion of the free energy in function of the radius of convergence of the power series I}) and from the definition of the  magnetization, one has:
$$M(\nu,c)=-\left(1+\frac{c\partial_c\rho_{\nu,c}}{\rho_{\nu,c}}\right).$$
In the proof of Lemma \ref{lemm : study of parameter S for c close to 1} we computed the first term of the Puiseux expansion of $\rho_{\nu,c}$ at $c=1$. Hence, by plugging it into this expression of $M(\nu,c)$, we study its asymptotic behavior near $c=1$. The constant coefficient equals $M(\nu,1)$, which is exactly the spontaneous magnetization. 

\bigskip

We end with the asymptotic behavior of $\chi(\nu,1)$ when $\nu\rightarrow \nu_\star\hspace{0.05cm}^{-}$.
As previously, from Equation (\ref{equa: Expresion of the free energy in function of the radius of convergence of the power series I}) and from the definition of the  susceptibility, one has:
$$\chi(\nu,c)= \left(\frac{c\partial_c\rho_{\nu,c}}{\rho_{\nu,c}}\right)^2 - \frac{c\partial_c\rho_{\nu,c}}{\rho_{\nu,c}} - \frac{c^2\partial_c^2\rho_{\nu,c}}{\rho_{\nu,c}}$$

We end the proof by plugging the Puiseux expansion of $\rho_{\nu,c}$ at $c=1$ computed in Lemma \ref{lemm : study of parameter S for c close to 1} into this expression. Recall from this lemma, that for  $0<\nu\leq\nu_\star$, the radius of convergence $\rho_{\nu,c}$ can be expressed as a power series in $(1-c)$, and so does $\chi(\nu,c)$. It converges at $c=1$, except for $\nu=\nu_\star$, and we obtain the expression of the susceptibility announced in the statement. For  $\nu>\nu_\star$, the radius of convergence $\rho_{\nu,c}$ can be expressed as a power series in $\sqrt{\vert 1-c\vert}$. By direct computation $\chi(\nu,c)$ can be expressed as a Puiseux series at $c=1$. Its first coefficient is of order $-\frac{1}{2}$, so that $\chi(\nu,c)$ tends to infinity when $c$ tends to $1$.
\end{proof}

\bibliographystyle{plainnat} 
\bibliography{IsingTetraMagn}  

\end{document}